\documentclass{amsart}

\usepackage{amsmath,amssymb,latexsym, amsthm, amsfonts,amstext,txfonts}
\usepackage{diagrams}

\begin{document}

\title[Models For The Maclaurin Tower Via A Derived Yoneda Embedding]{Models For The Maclaurin Tower Of A Simplicial Functor Via A Derived Yoneda Embedding }
\author{Peter J. Oman}

\address{Mathematics Department,  University of Western Ontario,
London, Ont., N6A 5B7} 
\email{poman@uwo.ca}

\begin{abstract} 
We prove that the Goodwillie tower of a weak equivalence preserving functor from spaces to spectra can be expressed in terms of the tower for stable mapping spaces.   Our proof is motivated by interpreting the functors $P_n$ and $D_n$ as pseudo-differential operators which suggests certain `integral' presentations based on a derived Yoneda embedding.  These models allow one to extend computational tools available for the tower of stable mapping spaces.  As an application we give a classical expression for the derivative over the basepoint. 
\end{abstract}


\maketitle
\newtheorem{thm}{Theorem}
\newtheorem{lem}[thm]{Lemma}
\newtheorem{prop}[thm]{Proposition}
\newtheorem{cor}[thm]{Corollary}
\theoremstyle{definition} \newtheorem{defn}[thm]{Definition}

\section{Introduction}
Goodwillie's homotopy calculus is a powerful tool in abstract homotopy theory.  It displays  homotopy theoretic information of a weak equivalence preserving functor, or homotopy functor, in a convenient form.   These ideas have provided a framework for arguments used in Waldhausen's K-Theory \cite{CalcI} and were applied by Kuhn in the chromatic setting \cite{gwtowers}.  In the unstable world, the tower interpolates between unstable and stable homotopy theory and has been exploited by Arone and Mahowald in the calculations of periodic homotopy groups of spheres \cite{perhom}.   

Homotopy calculus provides a non-vacuous 
language based in diffential geometry to study problems in homotopy theory.  
The goal of this paper is to exploit this analogy and give a simple model for the  Maclaurin tower of a finitary homotopy functor from spaces to spectra [such functors are called {\it functionals}].  Here the category of spaces means simplicial sets, $\mathbf{S}$, with its usual model structure and spectra to symmetric spectra, $\mathbf{Sp}$, with the stable model structure \cite{symspec}.  

The $N^{th}$ excisive approximation functor, $P_N$, corresponds to the operator mapping a function to its $N$-jet.  Viewing this functor as a `pseudo-differential operators' suggests a type of integral formula: the $N^{th}$ excisive
approximation of an arbitrary functor should be presentable as the integral
of a product with some `kernel' or `symbol'.
Using the dictionary between the smooth and homotopical settings, we can identify the appropriate analogue of linearity as homotopy colimit
preserving; thus, the
kernel should be determined by evaluating these operators on derived stable mapping
spaces.  These ideas culminate in \S 4 giving the following derived coend formulas for an arbitrary simplicial homotopy functor, corresponding to integration.
  
\theoremstyle{theorem}\newtheorem*{theoremA}{Theorem A} 
 \begin{theoremA}
The $n^{th}$ excisive approximation, $P_{n}$, and $n$-homogeneous approximation, $D_{n}$, (simplicial) functors acting on the category of functionals are pointwise weakly equivalent to the (enriched) derived coends \begin{center}
$(P_{n}F)(X) \sim \int^{\mathbf{S}^\circ} P_{n}( \Sigma^{\infty}Map_{S}(K,-)_{+})(X) \tilde{\wedge} F(K)$\\
$(D_{n}F)(X) \sim \int^{\mathbf{S}^\circ} D_{n} (\Sigma^{\infty}Map_{S}(K,-)_{+})(X) \tilde{\wedge} F(K)$.\end{center}
\end{theoremA}

\noindent There are essentially two (related) technical difficulties in obtaining this theorem: 
\begin{enumerate}
\item the {\it ad hoc} nature of the categorical framework at the foundations of homotopy calculus, \item obtaining derived versions of familiar categorical statements, such as a derived Yoneda embedding. 
\end{enumerate}
To elaborate on the former,  the basic objects in calculus are finitary homotopy functors from based spaces to spectra. This category can be awkward to work
with due to the fact it is ill-suited for categorical constructions, e.g., limits
and colimits. In addition, the natural homotopy theories encountered in calculus are difficult to work with. Indeed, these are not model structures but weaker homotopical structures.
The solution to this is a familiar application of the yoga of simplicial enrichment.  Following Biedermann et al.\cite{Biedermann}, we replace homotopy functors, and related categories, with categories of enriched presheaves of symmetric spectra and established appropriate model structures. Since we concentrate on finitary functors, our model structures are familiar and well behaved and in particular are cofibrantly generated.  Having simplicial model categories at our disposal allows us to solve the latter problem as well.  
In \S 3 we use some standard results in 2-category theory and enriched category theory, in particular the enriched Yoneda lemma, to indirectly obtain a very general (derived) spectrum-enriched Yoneda embedding.  In \S 4 we are interested in a  particular case of this lemma which takes the following form.

\newtheorem*{lemB}{Lemma B}
\begin{lemB}
Let $\mathbf{S}_{f}$ denote a skeleton of finite simplicial sets and $\mathbf{Sp}$ the category of symmetric spectra.  The category of homotopy-cocontinuous 
endofunctors acting on the  category of 
functionals, $h[\mathbf{S}_f,\mathbf{Sp}]$, is equivalent to the category of finitary homotopy functors 
$h[\mathbf{S}_f^{op} \otimes \mathbf{S}_f, \mathbf{Sp}].$ 
\end{lemB}

\noindent The theorem follows immediately after noting that, when acting on functionals, $P_N$ and $D_N$ preserve all homotopy colimits.

Having expressed the tower of an arbitrary functional in terms of stable mapping spaces, we note Arone \cite{Snaith} has developed computationally convenient models for the tower of stable mapping spaces.  
This suggests that many of the computational techniques particular to stable mapping spaces in \cite{fine}, extend to the tower for an arbitrary functional.
In addition, for $n = 1$, we deduce that the derivative of an arbitrary functional is given by a classical expression.

\theoremstyle{theorem}\newtheorem*{theoremC}{Theorem C}
\begin{theoremC}
Let $K^*$ be the S-dual of a finite simplicial set K, then the derivative at the basepoint of an arbitrary functional $F$ is given by the expression
$$\partial^{1}(F)(*) \sim \int^{\mathbf{S}^\circ} K^{*}\tilde{\wedge}F (K).$$
\end{theoremC}

More philosophically, this work presents an example of the utility of enriched categorical ideas within the foundations of homotopy calculus.  As mentioned above, the geometric intuition behind Theorem A requires a derived version of a familiar categorical statement.  To prove these derived statements, we typically use the 
well-known idea of including `higher homotopies' into our categorical data.  This should provide a familiar categorical language when objects (or categories) have some `up-to-something coherent' properties.   We use model categories enriched over simplicial sets and symmetric spectra as a realization of this idea, since retaining a model structure helps with homotopical calculations.  Goodwillie has suggested defining a category whose objects are homotopical categories (thought of as manifolds) with enough structure such that one can make `coordinate-free'  geometric constructions.   This work suggests some notion of enriched category theory should certainly play a role in this construction. 

Some comments about notation and language are appropriate.  
The notation employed follows Kelly \cite{Kelly}; for example, the category of (enriched) spectrum valued presheaves is denoted $[ \mathbf{S}, \mathbf{Sp}]$.  
We frequently use the same notation for a $\mathbf{V}$-enriched category (or construction) and its underlying $\mathbf{Set}$-enriched version.
The word derived is used ambiguously when describing  constructions rendering a functor homotopy invariant in some universal manner.  On occasion, these are not the usual left or right derived convention. For example, if F is homotopy invariant on fibrant-cofibrant objects, then precomposition with fibrant-cofibrant replacement is the derived functor of F.  

\section{Preliminaries}

\subsection{Category Theory}

We will assume that the reader is familiar with the basics of enriched category theory over a (weak) symmetric monoidal category.  To the reader unfamiliar with the notion of enrichment, it is generally safe to assume we are simply using a diagrammatic reformulation of the properties of hom-sets and generalize to hom-objects in a monoidal category.  Indeed, because the monoidal categories of interest are complete, cocomplete and closed the basic canon of categorical constructions and theorems hold verbatim \cite{Kelly}.  We  freely use the language of bicategories, or weak 2-categories, to naturally discuss collections of categories, functors, and natural transformations \cite{maclane}.  By a $\mathbf{V}$-category we mean an object of the  bicategory $\mathbf{V}$-$\mathbf{CAT}$ of (locally small) categories enriched over a monoidal category $\mathbf{V}$.  We circumvent any set theoretic objections to large categories like \textbf{Sets}, \textbf{S}, and $\mathbf{CAT}$ by appealing to Grothendieck's universe axiom \cite{Groth}. 

The crucial result we will need is the enriched version of the classical Yoneda embedding.  Let $Y:\mathbf{A}^{op} \rightarrow [\mathbf{A}, \mathbf{V}]$ be the ($\mathbf{V}$-enriched) Yoneda map between $\mathbf{V}$-categories which acts on objects by the map $A \mapsto \underline{hom}_{\mathbf{A}}(A,-)$.

\begin{prop}(\cite{Kelly} \S 4.4)
Let $Cocont[[\mathbf{A},\mathbf{V}],\mathbf{B}]$ denote the full subcategory of \\
$[[\mathbf{A},\mathbf{V}],\mathbf{B}]$ consisting of cocontinuous (colimit preserving) functors.  If $\mathbf{B}$ is a complete and cocomplete $\mathbf{V}$-category, then the functor $[Y,1]$ induces an equivalence of categories 
 $$Cocont[[\mathbf{A},\mathbf{V}],\mathbf{B}] \cong [\mathbf{A}^{op},\mathbf{B}]. $$
 
\end{prop}

\noindent We are interested in the special case of the above theorem when  $\mathbf{B}=[\mathbf{A},\mathbf{V}]$.

\begin{defn} An \textit{operator} is an object of the functor category 
$$Op^{\mathbf{V}}(\mathbf{A})\coloneqq Cocont[[\mathbf{A},\mathbf{V}],[\mathbf{A},\mathbf{V}]]$$. \end{defn} 

\begin{cor} (Classification of Operators)  The following categories are equivalent 
$$ Op^{\mathbf{V}}(\mathbf{A})\cong 
[\mathbf{A}^{op},[\mathbf{A},\mathbf{V}]] \cong [\mathbf{A}^{op} \otimes \mathbf{A}, \mathbf{V}].$$
\end{cor}

\begin{defn} 
The \textit{symbol} of an operator $\mathbb{L}\in Cocont[[\mathbf{A},\mathbf{V}],[\mathbf{A},\mathbf{V}]]$ is the corresponding bifunctor $Sym_{\mathbb{L}}:\mathbf{A}^{op} \otimes \mathbf{A}\rightarrow \mathbf{V}$. \end{defn}

The bicategory of monoidal categories together with (lax) monoidal functors and natural transformations will be denoted $\mathbf{Mon}$.  A monoidal functor $F:\mathbf{V} \rightarrow \mathbf{V'}$ induces a 2-functor $F:\mathbf{V}$-$\mathbf{CAT} \rightarrow \mathbf{V'}$-$\mathbf{CAT}$ by acting as the identity map on objects and on mapping hom-objects as $\underline{hom}(A,B) \mapsto F(\underline{hom}(A,B))$.  An adjunction in $\mathbf{Mon}$, 
$R \dashv L:\mathbf{V} \rightarrow \mathbf{V'}$, gives an equivalence of enriched functor categories $[ L\mathbf{I},\mathbf{C} ] \cong [\mathbf{I},R\mathbf{C}]$, where $\mathbf{I} \in \mathbf{V}$-$\mathbf{Cat}$ and $\mathbf{C} \in \mathbf{V'}$-$\mathbf{CAT}$.  Note the necessity of bicategories  to articulate the above categorical equivalence -- rather than a bijection of sets.

\subsection{Homotopy Theory}
A model category is the most common notion of a homotopy theory, but
it is only one of many ideas attempting to capture the relationship between topological spaces and its homotopy category.  Homotopy calculus, or simply calculus in this context, applies to homotopy functors between categories with a relatively weak notion of  homotopy theory.  The constructions reviewed in section 2.4 depend on the homotopy theory of homotopy functors -- which, in general, is \textit{not} a model category.  
The notion of homotopical category is  sufficiently weak  to discuss natural homotopy theories on functor categories. 
Homotopical categories consider  categorical localization from the point of view of `homotopy'.  This assumes some structure on the class of maps, called weak equivalences, we wish to invert and avoids the additional structures relating to cofibrations and fibrations.  The idea that a reasonable amount of homotopy theory can be developed in such categories stems from the observation that in a model category fibrations and cofibrations play a secondary role.   This approach has been developed in  \cite{DHKS} and we give the definition for the convenience of the reader.  

\begin{defn}
A \textit{homotopical category} will be a category $\mathbf{C}$ with a distinguished class $\mathcal{W}$, called weak equivalences, satisfying the following axioms:
\begin{enumerate}
\item $\mathcal{W}$ contains all the identity maps in $\mathbf{C}.$
\item $\mathcal{W}$ satisfies the two out of six property: for every three maps $f,g \text{ and } h$ such that the two compositions $gf \text{ and } hg$ are defined and are in $\mathcal{W}$ then so are the four maps $f,g,h, \text{ and } hgf$.
\end{enumerate}
\end{defn}

It is important to note that for \textit{any} small category \textbf{I} and \textit{any} homotopical category \textbf{C}, $[\mathbf{I}, \mathbf{C}]$ is naturally a homotopical category with pointwise weak equivalences.  Model categories are sometimes too refined an abstract notion of homotopy and homotopical categories give a coarser, internal theory.
The localization, or formal inversion, at weak equivalences will be called the homotopy category of \textbf{C}, denoted $Ho(\mathbf{C})$. In general, the construction of the homotopy category requires the existence of a larger universe.  For model categories, set theoretic obstructions are avoided as morphisms in the localized category are constructed as quotients.  The homotopical categories considered here are full subcategories of model categories and thus we can eschew subtle set theory.
Technically, we restrict ourselves to homotopical categories that are \textit{saturated} \cite{DHKS} and whose localizations exist in a given universe.

\subsection{Simplicial Model Categories}
The homotopy theory of homotopy functors, with pointwise equivalence, is notoriously difficult to express via model categories. However, simplicial enrichment typically gives `higher homotopy' information and simplicial functors are naturally compatible with this structure. 
Fortunately, many model structures are compatible with a natural enrichment over simplicial sets (see \cite{Simpstruct}).  

\begin{defn}
Let $\mathbf{M}$ be a model category which is enriched, tensored, and cotensored over simplicial sets with hom-space denoted $Map_{\mathbf{M}}(A,B)$.  We say $\mathbf{M}$ is a simplicial model category if, in addition, it satisfies the following axiom: \\

\textbf{(SM)} if $i:A \rightarrow B$ is a cofibration and $p:X \rightarrow Y$ is a fibration in $\mathbf{M}$, then the map of simplicial sets \begin{center}
 $Map_{\mathbf{M}}(B,X) \xrightarrow{i^{*}\times p_{*}} Map_{\mathbf{M}}(A,X) \times_{Map_{\mathbf{M}}(A,Y)} Map_{\mathbf{M}}(B,Y)$ \end{center}
is a fibration, which is a weak equivalence if either $i$ or $p$ is a weak equivalence.
\end{defn}

\noindent The canonical example of a simplicial model category is, of course, simplicial sets.  In addition, the simplicial model  structure on spectra \cite{symspec} will also play a special role.  
One advantage of simplicialization is the simplification of homotopically meaningful constructions, such as homotopy colimits.  

\begin{prop}
Let $\mathbf{M}$ be a simplicial model category.  For objects X and Y in \textbf{M} which are cofibrant and fibrant (respectively) and for any simplicial set K:
\begin{enumerate}
\item $X\otimes - \text{ and } -\otimes K$ preserve cofibrations and trivial cofibrations, 

\item $Map_{\mathbf{M}}(-,Y)$ converts cofibrations and trivial cofibrations into fibrations and trivial fibrations respectively,

\item $- \otimes -: \mathbf{M} \times \mathbf{S} \rightarrow \mathbf{M}$ preserves weak equivalences between pairs of cofibrant objects,
\item $Map_{\mathbf{M}}(-,-): \mathbf{M}^{op} \times \mathbf{M} \rightarrow \mathbf{S}$ preserves weak equivalences between pairs of fibrant-cofibrant objects.
\end{enumerate}
\end{prop}

As a corollary to the above proposition, simplicial model categories come equipped with a natural cylinder functor $- \otimes \Delta(1)$.
In general, we say maps $f$ and $g$ are simplicially homotopic if there exists a map $H:X \otimes \Delta(1) \rightarrow Y$ such that restriction to the initial and final vertex is f and g respectively.  The above lemma together with Whitehead's theorem implies that weak equivalences between fibrant-cofibrant objects are  simplicial equivalences; in other words, a simplicial functor by virtue of the compatibility with the simplicial structure is, in a sense, compatible with weak equivalences.
This observation is essential to developing a model category for the homotopy theory of homotopy functors. 

Before discussing a model structure for homotopy functors we must discuss diagram categories in the simplicial setting.
Model structures on  diagram categories, in general, are very difficult to construct. To guarantee a natural model structure we must make severe restrictions on the indexing category, e.g., assuming the indexing category is a Reedy \cite{Hovey} category or a very small category \cite{modcat}.  In the simplicial setting this becomes even more complicated.  However, if $\mathbf{M}$ has the additional property of cofibrant generation, then the projective structure does  generalize.  

\begin{thm}
 If $\mathbf{I}$ is a small, simplicial indexing category and $\mathbf{M}$ a simplicial, cofibrantly generated model category,  then $[\mathbf{I},\mathbf{M}]$ has a cofibrantly generated,  simplicial model structure with pointwise weak equivalences and fibrations.
 \end{thm}
 
 \begin{proof}
See \cite{Hirsch} \S 11.6 and simply use the word simplicial where appropriate.  For an enjoyable, general discussion of  homotopy theory for enriched diagram categories see \cite{EnrMod}.
\end{proof}

An important observation is that, a priori, (co)fibrant replacement is not \textit{simplicial}: the small object argument does not interact with the simplicial structure. Theorem 8 indirectly provides a simplicial replacement functor when the category is small. 

\begin{cor}
If $\mathbf{M}$ is that a small, cofibrantly generated simplicial model category, then there exists a simplicial  (co)fibrant  replacement functor.
\end{cor}

\begin{proof}
Choose a fibrant (cofibrant) replacement of the identity under the projective structure in $[\mathbf{M}, \mathbf{M}]$ and note that a cofibrant diagram is objectwise cofibrant
\end{proof}

\noindent Using lifting properties, Proposition 7 can be extended to the calculus of coends.

\begin{prop}(\cite{Hirsch} \S 18.4)
Let $\mathbf{M}$ be a simplicial model category and $\mathbf{I}$ a small category.
\item If $j: A \rightarrow B$ is an objectwise cofibration of $\mathbf{I}$-diagrams and $i:X \rightarrow Y$ is a cofibration of $\mathbf{I}^{op}$ diagrams of simplicial sets, then the induced map \begin{center}
$\int^{i \in \mathbf{I}}A \otimes Y \amalg_{\int^{i \in \mathbf{I}}A \otimes X} \int^{i\in I}B \otimes X \rightarrow \int^{i \in \mathbf{I}}B \otimes Y$ \end{center}
is a cofibration in $\mathbf{M}$ that is a weak equivalence if either i or j is an objectwise weak equivalence.

\end{prop}

\begin{cor}(\cite{Hirsch} \S 18.4)
Let $\mathbf{M}$ be a simplicial model category and $\mathbf{I}$ a small category.  
\begin{enumerate}
\item If F is a cofibrant functor $F:\mathbf{I}^{op} \rightarrow \mathbf{S}$ in the projective model structure and $j:G \rightarrow G'$ is a natural transformation which is an objectwise cofibration from $I \rightarrow \mathbf{M}$, then 
$\int^{i \in \mathbf{I}}F(i) \otimes G(i) \rightarrow \int^{i \in \mathbf{I}}F(i) \otimes G'(i)$ is a cofibration that is a weak equivalence if j is. 
 \item If F is a cofibrant functor $F:\mathbf{I}^{op} \rightarrow \mathbf{S}$ and G is an objectwise cofibrant functor from $\mathbf{I} \rightarrow \mathbf{M}$, then 
$\int^{i \in \mathbf{I}}F(i) \otimes G(i)$ is cofibrant.
\item If F is a cofibrant functor $F:\mathbf{I}^{op} \rightarrow \mathbf{S}$ and $j:G \rightarrow G'$ is a natural transformation which is an objectwise weak equivalence of objectwise cofibrant functors from $\mathbf{I} \rightarrow \mathbf{M}$, then 
$\int^{i \in I}F(i) \otimes G(i) \rightarrow \int^{i \in \mathbf{I}}F(i) \otimes G'(i)$ is a weak equivalence.
\item If G is a objectwise cofibrant functor $G:\mathbf{I} \rightarrow \mathbf{M}$ and $j:F \rightarrow F'$ is a natural transformation which is a weak equivalence of cofibrant functors from $\mathbf{I}^{op} \rightarrow \mathbf{S}$, then 
$\int^{i \in \mathbf{I}}F(i) \otimes G(i) \rightarrow \int^{i \in \mathbf{I}}F'(i) \otimes G(i)$ is a weak equivalence of cofibrant objects.
\end{enumerate}
\end{cor}

Let $Q_{c}:\mathbf{M}\rightarrow \mathbf{M}$ denote the simplicial cofibrant replacement functor in a simplicial model category $\mathbf{M}$,  $Q_{f}$  the simplicial fibrant replacement functor and $Q$ the simplicial fibrant-cofibrant replacement.
Corollary 11 can be used to deduce the well-known Bousfield-Kan model for  homotopy colimits in simplicial model categories.

\begin{thm}
If  $\mathbf{I}$ is an \textit{arbitrary} small category, then the homotopy  colimit of an $\mathbf{I}$ diagram $F$ is given by \begin{center}
$h\varinjlim_{\mathbf{I}} F(X) \sim \int^{i \in \mathbf{I}}Q_{c}F_{i}(X) \otimes N(i\downarrow \mathbf{I})^{op}$ \end{center}
\end{thm}

\begin{proof}
The homotopy properties of the above formulas follow from corollary 11 and noting that $N(-\downarrow \mathbf{I})^{op}$ is cofibrant  in $[\mathbf{I}^{op},\mathbf{S}]$.  With some additional work, we are able to show the desired universal property, see \cite{Hirsch} and \cite{EnrMod}.
\end{proof}

\noindent For our purposes, this local formula for the homotopy colimit is favorable over the typical Quillen left-derived notion as it allows us to immediately deduce many desirable properties.  For example, this formula implies homotopy  colimits in simplicial model categories share many of the same properties of those in $\mathbf{S}$ itself, e.g., the homotopy cofinality condition of \cite{Hirsch}. Familiar properties for homotopy colimits are essential for the generalization of homotopy calculus given in \cite{gwtowers}.  The dual statements regarding the homotopy properties of the cotensor are obvious.  Thus, one is able to express the homotopy limits as an end \\ 
$h\varprojlim_{\mathbf{I}} F(X) \sim \int_{i \in \mathbf{I}}Q_{f}F_{i}(X)^{N(i\downarrow \mathbf{I})}$.

\subsection{Homotopy Calculus}
Homotopy calculus is a powerful theoretical tool which displays the homotopy theoretic information of a functor in a convenient form.   Arone and Mahowald \cite{perhom} provide a striking example of its utility in solving fundamental problems in homotopy theory.
Here we present a quick overview of the basic notions developed by Goodwillie in \cite{CalcII}  and \cite{CalcIII}.  Although originally developed for functors from topological spaces to spectra, this machinery can be generalized to homotopical categories with a good notion of homotopy colimit and homotopy limit.  As pointed out by Kuhn \cite{gwtowers}, this includes cofibrantly generated simplicial model categories where we have explicit models for homotopy limits and colimits.  
The basic idea is to approximate homotopy functors between simplicial model categories by functors satisfying higher order excision properties.  The definition and construction of such functors depends on a calculi of cubical diagrams \cite{CalcII}.  We remind the reader of the basic definitions and results.
\begin{defn}
Let S be a finite set of cardinality n and \textbf{P}(S) the poset of all subsets considered as a category.  An \textit{n-cube} in a model category \textbf{C} is a functor $\chi:\mathbf{P}(S) \rightarrow \mathbf{C}$.  
\end{defn}
\noindent We will be comparing the initial space $\chi (\varnothing)$ with the homotopy limit of the functor restricted to $\chi:\mathbf{P}(S) - \varnothing=\mathbf{P}(S)_{0}$, as well as the dual notion. 

\begin{defn} 
\begin{enumerate}
\item A n-cube $\chi$ is \textit{cartesian} if the natural map $\chi(\varnothing) \rightarrow h\varprojlim_{P_{0}(S)} \chi$ is a weak equivalence.  
\item Let $\mathbf{P}(S)_{1}=\mathbf{P}(S)-S$,  an n-cube is \textit{cocartesian} if the natural map $h\varinjlim_{\mathbf{P}_{1}(S)}\chi \rightarrow \chi (S)$ is a weak equivalence.
\item If $\chi |_{\mathbf{P}(T)}$ is cocartesian for all $T \subset S$ and $|T| \ge 2$, then $\chi$ is \textit{strongly cocartesian}. 
\end{enumerate}
\end{defn}

\noindent Let \textbf{C}  be an arbitrary simplicial model category.

\begin{defn}
A homotopy functor $F:\mathbf{C} \rightarrow \mathbf{Sp}$ is \textit{n-excisive} if for every strongly cocartesian (n+1)-diagram $\chi:\mathbf{P}(S) \rightarrow \mathbf{C}$ the induced diagram of spectra $F \circ \chi$ is cartesian.
\end{defn}

\noindent The notion of cartesian here is unnecessary as $\mathbf{Sp}$ is stable, i.e., a cubical diagram is cartesian if and only if it is cocartesian.  Indeed, the calculus of functors with stable codomain only requires the notion of homotopy colimit.

 In \cite{CalcIII}, Goodwillie constructs a universal n-excisive approximation to F, denoted  $P_{n}F$, based on formal homotopical arguments. For example, if $F$ evaluated on the final object is weakly equivalent to the final object, $F(*)\sim *$, then $P_{1}F(X)=\Omega^{\infty}F(\Sigma^{\infty}X)$.  This is the analog of the Dold-Puppe stabilization in the homological setting (see \cite{linear}).  As n varies we form  the `Maclaurin tower' of a functor.

\begin{thm}( \cite{CalcIII} 1.13)
A homotopy functor $F:\mathbf{C} \rightarrow \mathbf{Sp}$ determines a tower of functors together with natural maps \\
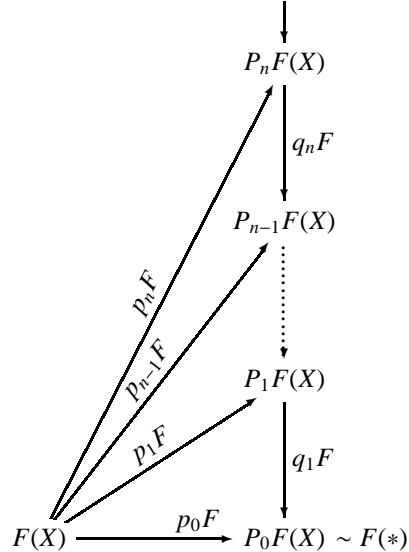
\begin{diagram}
	&&							&\dTo 					&	\\
	&&							&P_{n}F(X)				&	\\
	&&\ruTo(3,6)^{p_{n}F}			&\dTo_{q_{n}F} 			&	\\
	&&							&P_{n-1}F(X)				&	\\
	&&\ruTo(3,4)^{p_{n-1}F}			&\dDotsto 				&	\\
	&&					    		&P_{1}F(X) 				&	\\
	&&\ruTo(3,2)^{p_{1}F}			&\dTo_{q_{1}F}				& 	\\
F(X) &&\rTo^{p_{0}F}				&P_{0}F(X)				&\sim F(*) \\
\end{diagram}
inducing a natural map $q_{\infty}:F \rightarrow lim_{n}P_{n}F$.  The functor $P_{n}F$ is n-excisive and the map $p_{n}F:F \rightarrow P_{n}F$ is universal among such maps with domain F and codomain a n-excisive homotopy functor.  
\end{thm}

As a functor, $P_{n}$ is naturally simplicial and homotopy invariant when restricted to simplicial homotopy functors.
Denote the homotopy fiber of the map $p_{n}F:P_{n}F(X) \rightarrow P_{n-1}F(X)$ as $D_{n}F(X)$.  By construction, the functors $P_{n}$ and $D_{n}$ preserve \textit{all} homotopy colimits when $F$ has stable codomain.
 
 \begin{defn}
 We call a functor \textit{n-reduced} if $P_{n-1}F \sim *$ and \textit{n-homogeneous} if both n-reduced and n-excisive.
\end{defn}
 
\noindent It is easy to show $D_{n}F$ is always n-homogeneous and we use the term linear synonymously with 1-homogeneous   .

The notion of excisive readily extends to the simplicial model category $\mathbf{(C)^{n}}$, suggesting the notion of a multilinear functor.
Let $\mathcal{H}_{n}[\mathbf{C}, \mathbf{Sp}]$ denote the category of n-homogeneous functors with homotopical structure levelwise; analogously,
$\mathcal{L}_{n}[\mathbf{C},\mathbf{Sp}]$ is the category of symmetric n-multilinear functors with obvious homotopical structure.  Note the diagonal functor $\Delta_{n}:\mathcal{L}_{n}[\mathbf{C},\mathbf{Sp}]\rightarrow \mathcal{H}_{n}[\mathbf{C}, \mathbf{Sp}]$ is homotopy invariant.
When $\mathbf{C}$ is pointed as a model category, the cross-effect functor defines a homotopical inverse.

\begin{thm}(\cite{CalcIII} 3.5)
The functors\\
\begin{center}
$\Delta_{n}:\mathcal{L}_{n}[\mathbf{C},\mathbf{Sp}]\rightarrow \mathcal{H}_{n}[\mathbf{C}, \mathbf{Sp}]$\\
$cr_{n}:\mathcal{H}_{n}[\mathbf{C},\mathbf{Sp}]\rightarrow \mathcal{L}_{n}[\mathbf{C}, \mathbf{Sp}]$
\end{center}
naturally induce equivalences of homotopy categories.
 \end{thm}

Goodwillie \cite{CalcIII} refers to the multilinear functor corresponding to $D_{n}F$, denoted $D^{(n)}$F, as the differential.  Although a base point on the domain category is required to define the cross-effect, the following shows the classification still holds in the non-based case.  Let  $\phi:\mathbf{S}_{*} \rightarrow \mathbf{S}$
denote the forgetful functor.
\begin{thm}( \cite{CalcIII} 4.1)
The functors, induced by $\phi$,  \\
\begin{center}
$\phi^{*}:\mathcal{H}_{n}[\mathbf{S},\mathbf{Sp}] \rightarrow \mathcal{H}_{n}[\mathbf{S}_{*}, \mathbf{Sp}]$, \\
$\phi^{*}:\mathcal{L}_{n}[\mathbf{S},\mathbf{Sp}] \rightarrow \mathcal{L}_{n}[\mathbf{S}_{*},\mathbf{Sp}]$\end{center}
induce equivalences of homotopy categories.
\end{thm}

Symmetric multilinear functors naturally correspond to spectra with a $\Sigma_{n}$-symmetry provided the  functor is \textit{finitary}, i.e., determined by its value on finite subcomplexes.  Let $\Sigma_{n}$-$\mathbf{Sp}$ denote the category with objects spectra with a $\Sigma_{n}$-action and morphisms equivariant maps.  Define a morphism of  $\Sigma_{n}$-spectra to be a weak equivalence if as a morphism of spectra it is a weak equivalence.
 
 \begin{thm}(\cite{CalcIII})
Evaluation on 0-spheres defines an equivalence between the homotopy category of symmetric, n-multilinear functors from finite based spaces to spectra, with pointwise equivalences and $\Sigma_{n}$-$\mathbf{Sp}$.   In particular,
$$L(X_{1},...X_{n}) \sim C\wedge X_{1} \wedge ... \wedge X_{n}$$ where $C=L(S^{0},...,S^{0})$ is the coefficient $\Sigma_{n}$-spectrum.
\end{thm}

The layers in a tower of a finitary functor from based simplicial sets to spectra are given by $D_{n} F(X) \sim D^{n} F(X,...,X)_{h \Sigma_{n}}\sim (\partial^{n}F(*) \wedge (X^{\wedge^{n}}))_{h\Sigma_{n}}$,  where the notation is suggestive of its link to calculus in the smooth setting.  The dictionary with the smooth setting is a useful tool in that it allows one to apply geometric intuition to problems in homotopy theory, e.g., the partial chain rule established in \cite{Chain}.  The spectrum $P_{n}F(X)$ is thought of as the $n^{th}$-Taylor approximation about the basepoint evaluated at X.  As functors themselves, $P_{n}$ and $D_{n}$ are thought of as differential operators --where linearity is expressed categorically as homotopy cocontinuity and the degree as the order of excision.  The geometric analogy is the intuition for the classification given in Lemma B.

\section{A Derived {\bf Sp}-Enriched Yoneda Embedding}
 
\subsection{Homotopy Functors}

A simplicial functor preserves simplicial homotopies and hence weak equivalences between cofibrant-fibrant objects.  In other words, simply the property of being a simplicial functor between simplicial model categories implies a certain compatibility with the underlying homotopy theory.  Biedermann, Chorny, and R{\"o}ndigs \cite{Biedermann} use this observation to produce a model category whose homotopy theory models that of homotopy functors.  Their construction is a localization (or sheafification) of simplicial presheaves under the projective structure.   Here we use this idea to construct a model category for the homotopy theory of functionals. We circumvent any set theoretic difficulties to ensure our model structures are cofibrantly generated by assuming our categories are small.

Let \textbf{N} be a cofibrantly generated simplicial model category and \textbf{M} a \textit{small} simplicial category.  Assume that \textbf{M} comes with its own notion of weak equivalences compatible with the enrichment, e.g., weak equivalences are simplicial on some deformation retract whose objects are  `good'.  If $\mathbf{M}^{\circ}$ denotes the full subcategory of good objects with natural homotopical structure, then \textit{all} objects in $[\mathbf{M}^{\circ},\mathbf{N}]$ preserve weak equivalences.  
Let $h[\mathbf{M},\mathbf{N}]$ denote the full subcategory of $[\mathbf{M},\mathbf{N}]$ such that the underlying functor of each object preserves weak equivalences. Give the category $[\mathbf{M}^{\circ},\mathbf{N}]$ the projective model structure and define a homotopical structure on $h[\mathbf{M},\mathbf{N}]$ using pointwise equivalences.
Since $\mathbf{M}$ is a small simplicial model category, the good objects are simply the fibrant-cofibrant objects and, by Corollary 9,  there exists a simplicial fibrant-cofibrant replacement functor $Q:\mathbf{M} \rightarrow \mathbf{M^{\circ}}$.  Let $P^{*}:h[\mathbf{M},\mathbf{N}] \rightarrow [\mathbf{M}^{\circ},\mathbf{N}]$ denote the functor induced by restriction and note that, although not a Quillen pair, the functors $P^*$ and $Q^*$ define an equivalence of homotopy categories.

\begin{lem}
 The functors 
$Q^{*}:h[\mathbf{M},\mathbf{N}] \rightarrow [\mathbf{M}^{\circ},\mathbf{N}]$ and $P^{*}:h[\mathbf{M},\mathbf{N}] \rightarrow [\mathbf{M}^{\circ},\mathbf{N}]$ 
induce an equivalence of homotopy categories.  
\end{lem}

\begin{proof}
It is trivial to check that restriction to fibrant-cofibrant objects and  precomposition with fibrant-cofibrant replacement induces an equivalence with the identity.
\end{proof}

The category $[\mathbf{M}^{\circ},\mathbf{N}]$ is a natural model for the homotopy theory of homotopy functors with a simple, explicit and cofibrantly generated model structure.  Note the smallness of the indexing category is essential; indeed, when $\mathbf{M}$ is large there exists an analogous construction but at the cost of cofibrant generation and functorial replacement \cite{largedia}. 

\subsection{Homotopy Operators}
The goal of this section is to demonstrate a classification of homotopy operators, as in Corollary 3 with $\mathbf{V}=\mathbf{Sp}$.  Along the way we will show the  e
simplicial categorical notions which captures the `up-to-homotopy' information of functors sharing the formal properties of $P_{n}$ and $D_{n}$. Let $\mathbf{M}$ be a small simplicial model category.

\begin{defn}
A \textit{homotopy operator} is a object in $[h[\mathbf{M},\mathbf{Sp}], h[\mathbf{M},\mathbf{Sp}]]$ whose underlying simplicial functor preserves weak equivalences and \textit{all} homotopy colimits.  We denote the full subcategory of homotopy operators as $hOp(\mathbf{M})$.
\end{defn}

\noindent Note that the category of homotopy operators, as a subcategory of functors between homotopical categories, has a natural  homotopical structure defined pointwise.

Assume that we have a simplicial fibrant-cofibrant replacement functor $Q:\mathbf{M}\rightarrow \mathbf{M}^{\circ}$ as in Lemma 21, e.g., $\mathbf{M}$ is cofibrantly generated.
Given a homotopy operator $\mathbb{L}^h \in hOp(\mathbf{M})$ we define its homotopy symbol, $hSym_{\mathbb{L}^h}$, as the composition 

$$\mathbf{M}^{op} \xrightarrow{Q^{*}\Sigma^{\infty}Map_{M^{\circ}}(-,?)} h[\mathbf{M},\mathbf{Sp}] \xrightarrow{\mathbb{L}^h} h[\mathbf{M},\mathbf{Sp}]$$
with the derived mapping space.
This induces a well defined  functor \begin{center}
$hSym: hOp(\mathbf{M}) \rightarrow h[\mathbf{M}^{op}\otimes \mathbf{M}, \mathbf{Sp}]$, \end{center}
given pointwise by the formula 
\begin{center}
$hSym_{\mathbb{L}^{h}}(A,B)=\mathbb{L}^{h}(Q^{*}\Sigma^{\infty}Map_{\mathbf{M}^{\circ}}(A,-))(B)=\mathbb{L}^{h}(\Sigma^{\infty}Map_{M^{\circ}}(QA,Q-)(B)$. \end{center}

The symbol functor is our candidate inverse to a `Yoneda embedding' using the derived mapping spectrum.  In order to show this is an equivalence we proceed indirectly and show that both categories, defined by `up-to-homotopy' properties, can be made strict using simplicial enrichment.  The case of homotopy functors has already been discussed in the previous section giving an equivalence of homotopy categories 
\begin{center}
$h[\mathbf{M}^{op} \otimes \mathbf{M},\mathbf{Sp}] \sim [\mathbf{({M}^{\circ})^{op}}\otimes \mathbf{M}^{\circ} , \mathbf{Sp}] $.\end{center}

Observe that the adjunction $\Omega^{\infty} \dashv \Sigma^{\infty}: \mathbf{S_{(*)}} \rightarrow \mathbf{Sp}$ is (lax) monoidal; indeed, the simplicial structure on symmetric spectra is actually induced by the 2-functor $\Omega^{\infty}:\mathbf{Sp}$-$\mathbf{Cat} \rightarrow \mathbf{S}$-$\mathbf{Cat}$, where $Map_{\mathbf{Sp}}(A,B)=\Omega^{\infty}\underline{hom}_{\mathbf{Sp}}(A,B)$.   We conclude that the \textit{category} of simplicial functors $[\mathbf{({M}^{\circ})^{op}}\otimes \mathbf{M}^{\circ} , \mathbf{Sp}]$ is equivalent to the \textit{category} of $\mathbf{Sp}$-enriched functors $[\Sigma^{\infty}\mathbf{({M}^{\circ})^{op}}\otimes\Sigma^{\infty}\mathbf{M}^{\circ} , \mathbf{Sp}]$. The enrichment over spectra is reflecting the stability of the homotopy theory.
By Corollary 3, with $\mathbf{V}=\mathbf{Sp}$, the category of bifunctors classifies strict $\mathbf{Sp}$-operators

$$[(\mathbf{M}^{\circ})^{op}\otimes \mathbf{M}^{\circ} , \mathbf{Sp}] \cong [\Sigma^{\infty}(\mathbf{M}^{\circ})^{op}\otimes \Sigma^{\infty}\mathbf{M}^{\circ}, \mathbf{Sp}] \cong  Op^{\mathbf{Sp}}(\mathbf{M^{\circ}}).$$

We denote the composite functor assigning a  $\mathbf{Sp}$-operator to a homotopy operator as
 $$S:hOp(\mathbf{M}) \rightarrow Op^{\mathbf{Sp}}(\mathbf{M}^{\circ}).$$  It will be convenient to be rather explicit and give the functor S pointwise by the formula 

\begin{center}
$(S\mathbb{L}^{h})F(X)=
\int^{K\in\mathbf{M}^{\circ}} P^{*}hSym_{\mathbb{L}^{h}}(K,-) \otimes F(K)$\\$=
\int^{K\in\mathbf{M}^{\circ}} (\mathbb{L}^{h}(\Sigma^{\infty}Map_{\mathbf{M}^{\circ}}(QK,Q-)))(X) \wedge F(K).$ 
\end{center}

\noindent A model structure on strict operators will be induced by the equivalence 
$$Op^{\mathbf{Sp}}(\mathbf{M}^{\circ})\cong [(\mathbf{M}^{\circ})^{op} \otimes \mathbf{M}^{\circ}, \mathbf{Sp}],$$
 where $[(\mathbf{M}^{\circ})^{op} \otimes \mathbf{M}^{\circ}, \mathbf{Sp}]$ has the projective structure.

We construct a homotopical inverse by associating a homotopy operator to a strict operator.
An operator $\mathbb{L} \in Op^{\mathbf{Sp}}(\mathbf{M^{\circ}})$ is given by the $\mathbf{Sp}$-enriched coend $$\mathbb{L}F(X)=\int^{K\in\mathbf{M^{\circ}}} Sym_{\mathbb{L}}(K,X) \otimes F(K),$$ where $F \in [\Sigma^{\infty}(\mathbf{M}^{\circ}),\mathbf{Sp}]$ and $Sym_{\mathbb{L}} \in [\Sigma^{\infty}(\mathbf{M}^{\circ})^{op}\otimes \Sigma^{\infty}\mathbf{M}^{\circ},\mathbf{Sp}]$.  To make this homotopy invariant and act on homotopy functors we simply take the left derived version acting on $P^{*}F$.  In more detail, consider the \textit{simplicial} functor
 defined by 
 \begin{center}
 $\int^{K\in\mathbf{M^{\circ}}} Q_{c}Sym_{\mathbb{L}}(K,Q-) \wedge Q_{c}\tilde{F}(K):=\int^{K\in\mathbf{M^{\circ}}} Sym_{\mathbb{L}}(K,-) \tilde{\otimes} F(K)$, 
\end{center} 
where $Q_{c}$ denotes the  cofibrant replacement functor in the appropriate simplicial \textit{diagram} category.
This is clearly a homotopy functor as it is the image of a simplicial functor under $Q^{*}$.  As a functor of F it preserves weak equivalences  and all homotopy colimits -- this follows from expressing the homotopy colimit using the Bousfield-Kan model and Fubini's theorem for iterated coends. 
We also have a well defined functor $$D:Op^{\mathbf{Sp}}(\mathbf{M^{\circ}}) \rightarrow hOp(\mathbf{M})$$ which we call the \textit{derived} operator$$D(\mathbb{L})(F)(X) = \int^{\mathbf{M}^{\circ}} Sym_{\mathbb{L}}(K,X)\tilde{\otimes} F(X).$$

In summary, we have the following  diagram 
\newarrow{Congruent}33333
\begin{diagram}
[(\mathbf{M}^{\circ})^{op}\otimes \mathbf{M}^{\circ} , \mathbf{Sp}]    &	\cong		  &			Op^{\mathbf{Sp}}(\mathbf{M}^{\circ})	\\
\uTo^{P^{*}}\dTo_{Q^{*}}				&			 &	\uTo^{S}\dTo_{D}					\\
h[\mathbf{M}^{op} \otimes \mathbf{M}, \mathbf{Sp}]			&\lTo_{hSym} &hOp(\mathbf{M})							\\
\end{diagram}

\noindent and we will prove the homotopy symbol functor induces a homotopy equivalence by showing D and S induce equivalences of homotopy categories. 

\begin{lem}
If $\mathbb{L} \in Op^{\mathbf{Sp}}(\mathbf{M}^{\circ})$, then there is a natural equivalence $\mathbb{L} \sim SD(\mathbb{L})$.
\end{lem}

\begin{proof}
Note that strict operators are equivalent to their symbols.  In the projective model structure, weak equivalences between simplicial bifunctors are defined pointwise. 
The symbol of $SD\mathbb{L}$ is given by 
\begin{align*}
Sym_{SD\mathbb{L}}(A,B)&=   P^{*}hSym_{D \mathbb{L}} (A,B)  \\
&= P^{*}D \mathbb{L}(\Sigma^{\infty}Map_{\mathbf{M}^{\circ}}(QA,Q-))(B)\\
&= \int^{K\in\mathbf{M}^{\circ}}Q_{c}Sym_{\mathbb{L}}(K,B)\wedge Q_{c}\Sigma^{\infty}Map_{\mathbf{M}^{\circ}}(QA,QK).
\end{align*}
Since $Q_{c}Sym_{\mathbb{L}}(-,B)$ is a projectively cofibrant $(\mathbf{M}^{\circ})^{op}$ diagram and $Map_{\mathbf{M}^{\circ}}(QA,-)$ is objectwise cofibrant we have the  equivalences as desired
\begin{align*}
&\int^{K\in\mathbf{M}^{\circ}}Q_{c}Sym_{\mathbb{L}} (K,B) \wedge Q_{c}\Sigma^{\infty}Map_{\mathbf{M}^{\circ}}(QA,QK) \\
\sim & \int^{K\in\mathbf{M}^{\circ}}Q_{c}Sym_{\mathbb{L}} (K,B) \wedge \Sigma^{\infty}Map_{\mathbf{M}^{\circ}}(QA,K) \\
\sim & Q_{c}Sym_{\mathbb{L}}(QA,B)\\ \sim & Sym_{\mathbb{L}}(A,B).
\end{align*}

\end{proof}

\noindent We now show the converse for a homotopy operator $\mathbb{L}^{h} \in hOp(M)$
\begin{lem}
There is a natural equivalence  $\mathbb{L}^{h} \sim DS\mathbb{L}^{h}$
\end{lem}

\begin{proof}
The homotopy symbol of the composition is given by 
\begin{align*}
hSym_{DS \mathbb{L}^h}(A,B)=DS \mathbb{L}^h (\Sigma^{\infty}Map_{M^{\circ}}(QA,Q-))(B)&= \\
\int^{K\in\mathbf{M}^{\circ}}Sym_{S \mathbb{L}^h}(K,B) \tilde{\otimes} \Sigma^{\infty}Map_{M^{\circ}} (QA, QK) &\sim \\
\int^{K\in\mathbf{M}^{\circ}} Q_{c}Sym_{S\mathbb{L}^h}(K,B) \wedge \Sigma^{\infty}Map_{M^{\circ}} (QA, K) &\cong \\
Q_{c}Sym_{S \mathbb{L}^h}(QA,B)=Q_{c}hSym_{\mathbb{L}^h}(QA,B) & \sim  
hSym(A,B). & 
\end{align*}
We have shown an equivalence between the restriction of $DS \mathbb{L}$ and $\mathbb{L}$ on derived representables.   

Let $F \in h[\mathbf{M},\mathbf{Sp}]$ be an arbitrary homtopy functor, then $F \sim Q^{*}P^{*}F \sim Q^{*}F'$, where $F'$ is the cofibrant approximation to $P^{*}F \in [\mathbf{M}^{\circ},\mathbf{Sp}]$.  By the Yoneda lemma any functor $G \in [\mathbf{M}^{\circ}, \mathbf{Sp}]$ is a G-weighted colimit of representables.  Since $F'$ is cofibrant and representables are objectwise cofibrant, $F'$ is a homotopy colimit of representables.  It follows that any functor is equivalent to a homotopy colimit of derived representables proving the result.

\end{proof}
\begin{lem}
The previously defined functors induce equivalences of homotopy categories.
\begin{diagram}
Ho ([(\mathbf{M}^{\circ})^{op} \otimes \mathbf{M}^{\circ}, \mathbf{Sp}])   &\rTo^{\cong} & Ho( Op^{\mathbf{Sp}}(\mathbf{M}^{\circ}) )\\
\dTo_{\cong}				     &			   &\dTo_{\cong} \\
Ho( h[\mathbf{M}^{op} \otimes \mathbf{M}, \mathbf{Sp}])                    &\rTo^{\cong} &Ho( hOp(\mathbf{M})).
\end{diagram} \\
\end{lem}

\begin{lem}
Homotopy operators are classified by their symbols, i.e., the homotopy symbol functor induces an equivalence of homotopy categories\\
\begin{center}
 $hOp(\mathbf{M}) \xrightarrow{\sim} h[\mathbf{M}^{op}\otimes \mathbf{M}, \mathbf{Sp}].$
\end{center}
\end{lem}

We have no doubt a more direct proof of Lemma 26 is possible, but the use of simplicial categories seems to be of intrinsic interest.  For all $X \in \mathbf{M}$, $P_{n}$ is defined on the category of `local' functors $h[\mathbf{M}\downarrow X, \mathbf{Sp}]$ and is compatible, in some sense, with weak equivalences in $\mathbf{M}$ \cite{CalcIII}.  This suggests a generalization of homotopy operator which includes local information.  The use of the strict categorical equivalence of Corollary 3 hints at a classification of these `differential operators' where an up-to-homotopy approach seems to become exceedingly difficult.

\section{Applications to Calculus}
\subsection{The Maclaurin Tower}
The functors $P_{n}$ and $D_n$ are the motivating examples of homotopy operators, and Lemma 26 suggests that a given operator can be recovered from its restriction to representables.  
Unfortunately, we cannot apply Lemma 26 verbatim since simplicial sets is a large category.  This is not 
a serious problem since we will restrict to operators acting on finitary functors.  Finitary functors are extensions of functors defined on a small skeletal subcategory of finite simplicial sets $\mathbf{S}_{fin}$: the category of finitary homotopy functors $h[\mathbf{S},\mathbf{Sp}]$ is equivalent to $h[\mathbf{S}_{fin},\mathbf{Sp}]$.  Upon inspection  of Lemma 26  all that is required is a deformation of  simplicial sets to a subcategory of fibrant objects.  

The construction of such a functor follows as before: take the fibrant replacement  $Q$ of the identity functor in the projective model structure on the category $[\mathbf{S},\mathbf{S}]$.  Since $\mathbf{S}$ is large the treatment given here is insufficient to make sense of a projective structure on $[\mathbf{S}, \mathbf{S}]$. However, it is known a model structure on an appropriate subcategory, sufficient for  
our purpose, does exist and refer the reader to \cite{Chorny} and \cite{largedia} for the details.

Let $\mathbf{S}_{s}$ denote the full subcategory of $\mathbf{S}$ with object \textit{set}
 $\lbrace Q^{i}(X) | i \in \mathbb{N}, X \in \mathbf{S}_{fin} \rbrace$.
 Of course, $\mathbf{S}_{s}$ is not a simplicial model category but as a full subcategory of $\mathbf{S}$ it is homotopical and has a compatible simplicial enrichment.   The category $h[\mathbf{S}_{s},\mathbf{Sp}]$ models finitary homotopy functors.
The restricted functor $Q:\mathbf{S}_{s} \rightarrow \mathbf{S}_{s}$ serves as a deformation of $\mathbf{S}_{s}$ onto a category of good objects.  Let $\mathbf{S^{\circ}}$ denote the full subcategory of objects $\lbrace Q(X) | X \in \mathbf{S}_{s} \rbrace$.  The proof of the following follows Lemma 21 verbatim.

\begin{lem}
The functors $Q^{*}$ and $P^{*}$ induce an equivalence of homotopy categories
$[\mathbf{S}^{\circ}, \mathbf{Sp}] \sim h[\mathbf{S}_{fin}, \mathbf{Sp}]$.
\end{lem}

\noindent Similarly, the proof of Lemma 25 can be trivially adjusted to this case.

\begin{lemB}
Homotopy Operators are classified by their symbols,
i.e., the diagram
\begin{diagram}
[(\mathbf{S}^{\circ})^{op}\otimes \mathbf{S}^{\circ}, \mathbf{Sp}]&\rTo&Op^{\mathbf{Sp}}(\mathbf{S^{\circ}})\\
\dTo&&\dTo\\
h[\mathbf{S}^{op}_{s}\otimes \mathbf{S}_{s}, \mathbf{Sp}]&\rTo&hOp^{\mathbf{S}}(\mathbf{S}_{s})
\end{diagram}
induces an equivalence of homotopy categories.

\end{lemB}

\begin{theoremA}
The operators $P_{n}$ and $D_{n}$ acting on the category of functionals are pointwise weakly equivalent to the derived coends \begin{center}
$(P_{n}F)(X) \sim \int^{ S^{\circ}} P_{n}( \Sigma^{\infty}Map_{S}(K,-)_{+})(X) \tilde{\wedge} F(K)$\\
$(D_{n}F)(X) \sim \int^{ S^{\circ}} D_{n} (\Sigma^{\infty}Map_{S}(K,-)_{+})(X) \tilde{\wedge} F(K)$.\end{center}
\end{theoremA}

\noindent The corresponding result for based spaces is obvious.\\

The calculation of the n-excisive approximation of a functor has been reduced to the special case of its value on representables. 
For based, finite simplicial sets the coefficients of the homogeneous functors $D_{n}(\Sigma^{\infty}Map_{\mathbf{S}_{*}}(K,-))$ are given in (\cite{CalcIII} \S 7); for an alternative derivation see Klein and Rognes \cite{Chain} and Hesselholt \cite{Hessel}. 
In fact, the entire Maclaurin tower for stable mapping spaces has configuration space models developed by Arone \cite{Snaith}.  Combining these results gives a description of the  Maclaurin tower for an arbitrary finitary homotopy functor.
We do not intend to present Arone's models in all cases but will give an example.  Let $K \mapsto K^{*}$ be any simplicial model for the S-dual functor, for example $K \mapsto \underline{hom}_{\mathbf{Sp}}(\Sigma^{\infty}K,\Sigma^{\infty}S^{0})$.

\begin{lem}
Let K be a finite simplicial set, then $D_{1}(\Sigma^{\infty} Map_{\mathbf{S}_{*}}(K,-))(X) \sim K^{*}\wedge X$
and 
\begin{align*}
D_{1}(F)(X) &\sim \int^{ \mathbf{S^{\circ}_{*}}}  D_{1} (\Sigma^{\infty}Map_{\mathbf{S}_{*}}(K,-))(X) \tilde{\wedge} F(K) \\
&\sim \int^{ \mathbf{S}^{\circ}_{*}} (K^{*}\wedge X) \tilde{\wedge} F(K).
\end{align*}
\end{lem}

\noindent Thus, the derivative of a finitary simplicial homotopy functor is given by the simple, classical expression.

\begin{theoremC}
Let $K^*$ be the S-dual of a finite simplicial set K, then the derivative at the basepoint of an arbitrary functional $F$ is given by the expression
$$\partial^{1}(F)(*) \sim \int^{\mathbf{S}^{\circ}} K^{*}\tilde{\wedge}F (K).$$
\end{theoremC}

\section*{Acknowledgment}  This paper presents some of the results of my Ph.D. thesis (Brown University, 2007).  I would like to thank my advisor Tom Goodwillie for  generously and patiently sharing his knowledge with me while I worked on this project.

\end{document}